\documentclass[12pt]{amsart}
\usepackage{amssymb,amsthm,amsmath,epsfig,latexsym,calc}

\newlength{\bracewidth}
\newcommand{\myunderbrace}[2]{\settowidth{\bracewidth}{$#1$}#1\hspace*{-1\bracewidth}\smash{\underbrace{\makebox{\phantom{$#1$}}}_{#2}}}

\newcommand{\mmbox}[1]{\mbox{${#1}$}}
\newcommand{\proj}[1]{\mmbox{{\mathbb P}^{#1}}}

\newcommand{\K}{\mathbb{K}}

\newcommand{\rank}{\mathop{\rm rank}\nolimits}

\numberwithin{equation}{section}
\newtheorem{theorem}{Theorem}[section]
\newtheorem{lemma}[theorem]{Lemma}

\newtheorem{corollary}[theorem]{Corollary}
\newtheorem{conjecture}[theorem]{Conjecture}

\newtheorem{question}[theorem]{Question}

\theoremstyle{definition}
\newtheorem{definition}[theorem]{Definition}
\newtheorem{remark}[theorem]{Remark}
\newtheorem{example}[theorem]{Example}

\begin{document}


\title {Bounding invariants of fat points using a coding theory
construction}

\author{\c Stefan O. Toh\v aneanu}
\address{Department of Mathematics\\
The University of Western Ontario\\
London, ON N6A 5B7, Canada}
\email{stohanea@uwo.ca}
\urladdr{http://www.math.uwo.ca/$\sim$stohanea/}

\author{Adam Van Tuyl}
\address{Department of Mathematical Sciences \\
Lakehead University \\
Thunder Bay, ON P7B 5E1, Canada}
\email{avantuyl@lakeheadu.ca}
\urladdr{http://flash.lakeheadu.ca/$\sim$avantuyl/}

\thanks{Version: Revised Feb. 19, 2012}

\begin{abstract}
Let $Z \subseteq \proj{n}$ be a fat points scheme, and let $d(Z)$ be the
minimum distance of the linear code constructed from $Z$. We show that
$d(Z)$ imposes constraints (i.e., upper bounds) on some specific shifts
in the graded minimal free resolution of $I_Z$, the defining ideal of $Z$.
We investigate this relation in the
case that the support of $Z$ is a complete intersection;
when $Z$ is reduced and a complete intersection we give lower bounds for $d(Z)$
that improve upon known bounds.
\end{abstract}
\maketitle


\section{Introduction and Notations}

Let $\K$ be a field of characteristic zero. Let
$X=\{P_1,\ldots,P_s\}\subset\proj{n}$
be a reduced set of points not all contained in a hyperplane. A {\it
fat point scheme}
 $Z$ in $\proj{n}$ with {\it support} $\operatorname{Supp}(Z)=X$, and denoted
 $$Z = m_1P_1+\cdots +m_sP_s$$ is the zero-dimensional scheme defined by
$I_Z = I^{m_1}_{P_1} \cap \cdots \cap I^{m_s}_{P_s}\subset
R=\K[x_0,\ldots,x_n]$,
where $I_{P_i}$ is the defining ideal of the point $P_i$.
The scheme $Z$ is sometimes called a {\it set of fat points}.
We call $m_i$ the
{\it multiplicity} of the point $P_i$. When all the $m_i$'s are equal,
we say $Z$
is {\it homogeneous}.

To a fat point scheme we associate a linear code with generating matrix
\[A(Z) =
\left[\begin{array}{ccc}
\myunderbrace{c_1 ~~ \cdots ~~ c_1}{m_1} & \cdots &
  \myunderbrace{c_s ~~\cdots~~ c_s}{m_s}\end{array}\right],
\vspace{.5cm}
\] where each $c_i$ is a column vector with entries equal to
the homogeneous coordinates of the point $P_i$. This linear code has
parameters $[m_1+\cdots+m_s,n+1,d]$, where $d$ denotes, as usual, the
minimum Hamming distance of the code. Depending on the situation, $d$
(denoted with $d(Z)$) will be called {\it the minimum distance of the
matrix $A(Z)$},
or {\it the minimum distance of the fat point scheme $Z$}.

Note that in the matrix $A(Z)$, if we replace a column $c_i$ with any
of its (nonzero)
scalar multiples, or if we permute in any way the columns of $A(Z)$,
the parameters
of this linear code do not change. As a consequence of this simple
observation
one can create a fat point scheme $Z$ from any generating matrix of any linear
code, by identifying the columns of this matrix with points (fat
points, if some
columns are proportional) in a projective space. Hansen \cite{h},
Gold, Little, and
Schenck \cite{gls} and the first author \cite{t3} took this approach
in the case
when $Z$ is reduced (i.e., $m_i=1$, or the generating matrix has no
proportional
columns) to obtain bounds on the minimum distance using homological algebra. In
particular, it was shown (\cite{gls,t1,t3}) that the minimum distance
$d$ can be
bounded below in terms of the graded
shifts appearing in the graded minimal free resolution of $R/I_Z$.
The lower bounds of $d=d(Z)$ can be interpreted as upper bounds for
the corresponding homological
 invariants of $Z$; our main goal is to consider the same problem, but
we drop the
condition that $Z$ is reduced.

Let $Z \subseteq \proj{n}$ be any zero-dimensional (arithmetically)
Cohen-Macaulay subscheme. Let $R=\K[x_0,\ldots,x_n]$, and let $I=I_Z$
be the ideal of
$Z$. Suppose that the graded minimal free resolution of $R/I$ has the form
\[0 \rightarrow F_{n} = \bigoplus_{j=1}^{u_n} R(-a_{j,n}) \rightarrow
\cdots \rightarrow F_1=\bigoplus_{j=1}^{u_1} R(-a_{j,1}) \rightarrow R
\rightarrow R/I \rightarrow 0. \]
We set $t_i = \min_j\{a_{j,i}\},$ for each $i=1,\ldots,n$, and we
call $s_n(Z)=t_n-n$ the {\it minimum socle degree} of $Z$.

\begin{remark}
We are abusing terminology slightly in the above definition.  Since
$R/I$ is a Cohen-Macaulay ring of dimension $1$, we can find a
linear form $L$ that is a non-zero divisor on $R/I$.  Then, the ring $A=
R/(I,L)$
is an Artinian ring.  The socle of this
ring is $\operatorname{soc}(A) = 0:\overline{m}$ where
$\overline{m} = \bigoplus_{j \geq 1} A_j$.  The socle is an ideal
of $A$;  it can be shown that the degrees of the minimal generators of
this ideal
are encoded into the graded shifts at the end of the graded minimal
resolution of $R/I$, but shifted by $n$ (see \cite{ku} for more details).
We call $s_n(Z)$ the
minimum socle degree to recognize this connection.
\end{remark}

When $Z$ is reduced, the homological lower bounds for the minimum distance
mentioned above are expressed in terms of the minimum socle degree, that is,
$$d(Z)\geq s_n(Z).$$
One can observe that the
minimum distance puts constraints on the shifts in the graded minimal free
resolution of $Z$: $d(Z)+n\geq t_n\geq t_{n-1}+1\geq\cdots\geq t_1+n-1$.
Example \ref{ci} shows that the minimum distance will never precisely determine
the minimum socle degree. Nevertheless, in \cite{t3} it is shown that
$d(Z) = s_n(Z)$
is attained for a family of examples due to J. Migliore.

The goal of this paper is to obtain similar constraints for the graded minimal
free resolution of a fat point scheme. Once we add multiplicities to
the points,
even when the support is as nice as possible (e.g., complete intersection),
the shifts in the resolution change with no visible pattern.

In this paper, we take the point-of-view of describing how
$d(Z)$, the minimum distance of a linear code constructed
from $Z$,
can be used to bound homological invariants of $I_Z$.  One could invert this
point-of-view by studying how homological invariants are related to
the minimum distance, and linear codes in general.
The references \cite{gls,h,t1} took this second approach.
Because some of our results do not require $char(\K)=0$,
these results could also be used to study linear codes.  We see our results
as complementing ongoing research to algebraically study linear codes.
For example, the references \cite{clo,tvn}
provide an entry point for readers who would like to learn more about
coding theory
from an  algebraic geometric perspective.  Feng, Rao, and Berg
\cite{frb} present a
version of B\'ezout's Theorem using resultants to address a lower bound for the
minimum distance of a special class of algebraic geometric codes. And
more recently,
Sarmienta, Vaz Pinto, and Villarreal \cite{svv} used algebraic methods to study
problems arising from coding theory.

Our paper is structured as follows. In Section 2.1 we present a short
introduction
 to the study of the minimum distance of a linear code. We consider the
relationship between the minimum distance of the linear code created
from a fat point
scheme $Z$ and the minimum distance of the linear code constructed from
$X=\operatorname{Supp}(Z)$ (Theorem \ref{crudebounds}). The bounds obtained are
optimal, as shown by examples. Since we have not found this result
in the literature, we decided to write down the detailed proof, even though the
result seems to be natural. In Section 2.2, Theorem \ref{hombound}
finds an upper
 bound for the first homological invariant of the fat point scheme $Z$ in terms
 of the minimum distance of the associated linear code. The invariant
considered
is $t_1$, the minimal degree of a generator of $I_Z$. In Section 3 we
present our
 main result, Theorem \ref{maintheorem}, which gives an upper bound
for the minimum
socle degree of the fat points scheme in terms of the minimum distance
of its support.
 The main tool used is the machinery of separators of fat points
developed in \cite{gmv}.
In Section 4, we specialize to the case that the support of
the fat points $Z$
in $\proj{n}$ is a complete intersection. In particular, we use
B\'ezout's Theorem in
Corollary \ref{bezout} to give new lower bound on $d(Z)$.

\noindent
{\bf Acknowledgments.}  All the computations were done using {\it
Macaulay 2} \cite{Mt}.
The second author was supported by a Discovery Grant
from NSERC.
We thank the anonymous referee for important comments and suggestions,
especially for
pointing out that the version of B\'ezout's Theorem that we require also holds
for algebraically closed fields of positive characteristic.


\section{Minimal degree of hypersurfaces containing $Z$}


\subsection{The minimum distance of a linear code}

Let $\K$ be any field, and let $n\geq 1$ and $s\geq n+1$ be two integers.
A {\em linear code} $\mathcal C$ of length $s$ and dimension $n+1$ is the
image of an injective $\K$-linear map $\phi:\K^{n+1}\rightarrow \K^s$.
The {\em minimum (Hamming) distance} $d$ of $\mathcal C$ is the minimum
number of nonzero entries in a nonzero element (codeword) in $\mathcal C$.
The numbers $s$, $n+1$ and $d$ are called {\em the parameters}
of $\mathcal C$, and the code is called an $[s,n+1,d]$-code.

Any matrix representation of $\phi$ is called a {\em generating matrix}. This
representation is an $(n+1)\times s$ matrix with entries in $\K$,
of rank $n+1$. Usually one writes this matrix representation
of $\phi$ in the standard bases of $\K^{n+1}$ and $\K^s$. We can reverse
the process: if $A$ is an $(n+1)\times s$ matrix of rank $n+1$, we can
create a linear code having $A$ as a generating matrix, by
constructing the map $\phi$ from $A$, using the standard bases.

Let $c_1,\ldots,c_s \in \K^{n+1}$ be
$s \geq n+1$ vectors which have the property
that no vector is a scalar multiple of another vector.
That is, for every $i \neq j$, $c_i \neq ec_j$ for every $0 \neq e \in
\K$.
Let
\[A(X) =
\begin{bmatrix}
c_1 & \cdots & c_s
\end{bmatrix},\]
and assume that the vectors $c_i$ have been picked so that
$\rank(A(X)) = n+1$.   We call this matrix {\it reduced}.
With the same vectors, fix integers $m_1,\ldots,m_s$, and set
\begin{equation}\label{standardform}
A(Z) =
\left[\begin{array}{ccc}
\myunderbrace{c_1 ~~ \cdots ~~ c_1}{m_1} & \cdots &
  \myunderbrace{c_s ~~\cdots~~ c_s}{m_s}\end{array}\right].
\newline
\vspace{.5cm}
\end{equation}
When $m_1 = \cdots = m_s = 1$, then $A(Z) = A(X)$.  If $m_i > 1$,
then we say $A(Z)$ is {\it non-reduced}, and we call $A(X)$ the {\it
reduced matrix} associated to $A(Z)$.

\begin{definition} If $A$ is an $(n+1)\times s$ matrix with entries in a
field $\K$
with rank $n+1$, then the {\it minimum distance of $A$} is
\[d(A) = \min\{d ~|~ \mbox{there exists $s-d$ columns of $A$ that span an
$n$-dimensional
space}\}.\]
\end{definition}

If $A=A(X)$, respectively, $A=A(Z)$, we will denote $d(A)$ by $d(X)$,
respectively $d(Z)$.

\begin{remark} If $\mathcal C\subseteq \K^s$ is a linear code with
parameters $[s,n+1,d]$,
then $d$ is equal to the minimum distance of any generating matrix of $\mathcal
C$, as we
defined above.  

We give a simple linear algebra argument for this fact.
Suppose that $$A=\left[
\begin{array}{cccc}
a_{1,1}&a_{1,2}&\cdots&a_{1,s} \\
\vdots&\vdots& &\vdots \\
a_{n+1,1}&a_{n+1,2}&\cdots&a_{n+1,s}
\end{array}
\right]$$ is a generating matrix for $\mathcal C$. Denote with 
$r_1,\ldots,r_{n+1}$ the rows of $A$, and with 
$c_1,\ldots,c_s$ the columns. It is enough to show that there 
exists a codeword $v$ (so a linear combination of the $r_i$'s) with 
$j$ zero entries in positions $i_1,\ldots,i_j$ if and only if 
the dimension of the vector space spanned by 
$c_{i_1},\ldots,c_{i_j}$ is $\leq n$.

$(\Rightarrow)$ Let $v=u_1r_1+\cdots+ u_{n+1}r_{n+1}, u_i\in \mathbb K$
be a codeword with the first $j$ entries equal to zero. 
This means that $$u_1a_{1,1}+\cdots+u_{n+1}a_{n+1,1}=0,~\ldots,~ 
u_1a_{1,j}+\cdots+u_{n+1}a_{n+1,j}=0.$$ 
Equivalently,
 $$c_1,\ldots,c_j\in\{(x_1,\ldots,x_{n+1})\in\mathbb K^{n+1}~|~u_1x_1+\cdots+u_{n+1}x_{n+1}=0\}.$$ So we have $j$ points of $\K^{n+1}$ on a hyperplane, which
implies
$$\dim \operatorname{span} \langle c_1,\ldots,c_j\rangle\leq n.$$

$(\Leftarrow)$ Suppose 
$\dim \operatorname{span} \langle c_1,\ldots,c_j\rangle\leq n$. 
If we consider the matrix whose $i$th row is given by $c_i$,  i.e.,
take the transpose of the matrix with columns given by then $c_i$'s,
then this matrix has rank at most $n$.  But
this means that the homogeneous system of 
equations in the variables $y_1,\ldots,y_{n+1}$
\begin{eqnarray}
a_{1,1}y_1+\cdots+a_{n+1,1}y_{n+1}&=&0\nonumber\\
&\vdots&\nonumber\\
a_{1,j}y_1+\cdots+a_{n+1,j}y_{n+1}&=&0\nonumber
\end{eqnarray} 
must have a nontrivial solution 
$(u_1,\ldots,u_{n+1})$. 
This now means that the codeword 
$v=u_1r_1+\cdots+u_{n+1}r_{n+1}$ has the first $j$ entries equal to zero.
\end{remark}

\begin{remark} Suppose we are given any $t \geq n+1$ vectors in
$a_1,\ldots,a_t \in \K^{n+1}$,
such that $A = \begin{bmatrix} a_1 & \cdots & a_t \end{bmatrix}$ has
rank $n+1$.  Then,  rescaling any proportional vectors or permuting
columns of $A$ does not change the value of the minimum distance $d(A)$
for
this matrix.
\end{remark}

The value of $d(Z)$ is related to $d(X)$ as follows:

\begin{theorem} \label{crudebounds}
Let $A(Z)$ be a matrix of the form $(\ref{standardform})$ and
assume that the columns of $A(Z)$ have also
been permuted so that $m_1 \geq m_2 \geq \cdots \geq m_s$.
If $A(X)$ is the reduced matrix associated to $A(Z)$,  and $d(X) = d$,
then
\[m_1+\cdots+m_d \geq d(Z) \geq m_{s-d+1} + \cdots + m_s. \]
In addition, if $m_1 = \cdots = m_s = m$, then $d(Z) = md(X)$.
\end{theorem}

\begin{proof}
Let
\[\Lambda = \left\{W = \{c_{i_1},\ldots,c_{i_{e}}\} \subseteq
\{c_1,\ldots,c_s\} ~\left|~ \dim (\mbox{span} \langle W \rangle) = n
\right\}\right.,\]
i.e., $\Lambda$ is the collection of all $e$ columns that
span an $n$-dimensional space.  In particular, since $d(X) =d$,
we can find $s-d$ columns $c_{i_1},\ldots,c_{i_{s-d}}$ such
that $\{c_{i_1},\ldots,c_{i_{s-d}}\} \in \Lambda$.  So,
if $W \in  \Lambda$, then $n \leq |W| \leq s-d$, where
the first inequality comes from the fact that one needs
at least $n$ vectors to span an $n$-dimensional space.

Computing $d(Z)$ is equivalent to finding the maximum number
of column vectors of $A(Z)$ that span an $n$-dimensional space.
If any column of $A(Z)$ is used to span this $n$-dimensional
space, we should also take all copies of that column;  each
extra column does not contribute to the dimension (being
a scalar multiple of the first column) but it contributes
to the total number of columns being used.

Thus, if $M = m_1+\cdots+m_s$,
we can find $e$ distinct columns $c_{i_1},\ldots,c_{i_e}$
such that the following $M-d(Z)$ columns of $A(Z)$
\[\{\underbrace{c_{i_1},\ldots,c_{i_1}}_{m_{i_1}},\ldots,
  \underbrace{c_{i_e},\ldots,c_{i_e}}_{m_{i_e}}\}\]
span an $n$-dimensional subspace of $\K^{n+1}$.
But then $\{c_{i_1},\ldots,c_{i_{e}}\} \in \Lambda$.
Thus
\[M-d(Z) = \max\{{m_{i_1}+\cdots+m_{i_{e}}} ~|~ \{c_{i_1},\ldots,c_{i_e}\}
\in \Lambda \}.\]
So $M-d(Z)\geq m_{i_1}+\cdots+m_{i_e}$ for all
$W=\{c_{i_1},\ldots,c_{i_e}\} \in \Lambda$.
Because there must exist $W\in\Lambda$ with $|W|=s-d$, we obtain that
$$M-d(Z)\geq m_{i_1}+\cdots+m_{i_{s-d}},$$ for some
$i_1,\ldots,i_{s-d}\in[s]$.

If one has a (finite) decreasing (not necessarily strictly) sequence of
numbers, then the sum of any $k$ terms of this sequence is greater than or
equal
to the sum of the last $k$ terms of the sequence. In our case the sequence
is $m_1\geq m_2\geq\cdots\geq m_s$, and $k=s-d$. So $$M-d(Z)\geq
m_{d+1}+\cdots+m_s,$$
and therefore $$m_1+\cdots+m_d\geq d(Z).$$

Alternatively, we can write
\[d(Z) = \min\left\{m_{j_1}+\cdots+m_{j_{s-e}} ~\left|~
\begin{array}{c}
\{j_1,\ldots,j_{s-e}\} = [s] \setminus \{i_1,\ldots,i_{e}\} \\
\text{with}~~
\{c_{i_1},\ldots,c_{i_{e}}\} \in \Lambda
\end{array} \right\}\right..
\]
Because any $W \in \Lambda$ has $|W| \leq s-d$, the smallest
sum we can have contains $s-(s-d) = d$ terms.  Moreover,
since $m_1 \geq m_2 \geq \cdots
\geq m_s$, we must have $$d(Z) \geq \ m_{s-d+1} + \cdots + m_{s}.$$

When $m_1 = \cdots = m_s = m$, our two bounds give $md\geq d(Z)\geq md.$
\end{proof}

\begin{remark}
When $m_1 = \cdots = m_s = m$, then the corresponding linear code
is sometimes called a {\it $m$-fold repetition code} (see \cite{w}).
\end{remark}

\begin{example}
Both of the bounds of Theorem \ref{crudebounds} can be attained.
Consider the following two matrices with entries in $\mathbb{F}_2$, the
finite field with two elements:
\[A(Z_1) =
\begin{bmatrix}
1 & 1 & 1 & 0 & 0 & 0 & 0 & 0 & 0 \\
0 & 0 & 0 & 1 & 1 & 0 & 0 & 1 & 1 \\
0 & 0 & 0 & 0 & 0 & 1 & 1 & 1 & 1
\end{bmatrix}
~~\text{and}~~
A(Z_2) =
\begin{bmatrix}
0 & 0 & 1 & 0 & 0 \\
1 & 1 & 0 & 1 & 0 \\
1 & 1 & 0 & 0 & 1
\end{bmatrix}.\]
These two matrices have the same reduced associated matrix
\[A(X) =
\begin{bmatrix}
1 & 0 & 0 & 0\\
0 & 1 & 0 & 1\\
0 & 0 & 1 & 1
\end{bmatrix},\] which has $d=d(X)=1$.

For $A(Z_1)$ we have $d(Z_1)=3$, and $m_1=3, m_2=m_3=m_4=2$. So the upper
bound
in Theorem \ref{crudebounds} is attained.
For $A(Z_2)$, we have $d(Z_2)=1$, and $m_1=2, m_2=m_3=m_4=1$.
In this case, the lower bound in Theorem \ref{crudebounds} is attained.
\end{example}


\subsection{An upper bound for $t_1$} Let $\K$ be any field and let
$Z\subset\mathbb P_{\K}^n$ be a fat point scheme defined by $I_Z\subset
R=\mathbb K[x_0,\ldots,x_n]$. Let $X$ be the support of $Z$, so $X$ is a
reduced finite set of points. Assume that they are not all contained
in a hyperplane in $\proj{n}$.

As described above, let $A(X)$ be the reduced matrix associated to
 $X$ and let $A(Z)$ the non-reduced matrix associated to $Z$. Let $d(X)$, and
$d(Z)$ respectively, be the minimum distances of these matrices.

\begin{remark} \label{geometric} We can reinterpret
$d(X)$ as a geometric condition.    Let $hyp(X)$ denote the  maximum
number of points of $X$ contained in a hyperplane.  Then
\[d(X) = |X| - hyp(X).\]
We can observe this by noting that the columns corresponding to the points
in the hyperplane must span a vector space of dimension $n$.
\end{remark}

The ring $R/I_Z$ has a graded minimal free resolution as given in the
introduction. Let $\alpha(I_Z) := t_1 = \min\{u ~|~ (I_Z)_u \neq 0\}$.
So $\alpha(I_Z)$ is the minimal degree of a hypersurface containing $Z$.

\begin{theorem}\label{hombound}
Let $Z = m_1P_1 + \cdots + m_sP_s \subseteq \proj{n}$ be a fat point
scheme. Set $m(Z) = \max\{m_1,\ldots,m_s\}$. Then
\[d(Z) \geq \alpha(I_Z) - m(Z).\]
\end{theorem}

\begin{proof}
We first consider the case that $A(Z) = A(X)$, i.e.,
$m_i =1$ for $i=1,\ldots,s$ for any integer $s \geq n+1$.
Suppose that $d(X) < \alpha(I_X) -1$.
Because $X$ does not lie in a hyperplane, $\alpha(I_X) \geq 2$. From
  Remark \ref{geometric}, there is a hyperplane of equation $H=0$ that
contains $s -d(X)$ of the points of $X$.  For the remaining $d(X)$ points,
say $Q_1,\ldots,Q_{d(X)}$,
let $L_i$ be any linear form in the ideal of the point $Q_i$.
Then the hypersurface defined by $G= H\cdot L_1\cdots L_{d(X)}$ passes
through all the points
of $X$, and $\deg G = 1 + d(X) < \alpha(I_X)$, a contradiction.
So $d(X) \geq \alpha(I_X) -1$.

We now proceed by induction on the tuple $(s,(m_1,\ldots,m_s))$,
that is, we assume that the statement holds for
all tuples of the form $(s,(a_1,\ldots,a_s))$ with
\[(\underbrace{1,\ldots,1}_s) \preceq (a_1,\ldots,a_s) \prec
(m_1,\ldots,m_s),\]
or for all tuples of the form $(s-1,(a_1,\ldots,\hat{a}_i,\ldots,a_s))$
where $s-1 \geq n+1$ and
$$(\underbrace{1,\ldots,1}_{s-1})
\preceq (a_1,\ldots,\hat{a}_i,\ldots,a_s) \prec
(m_1,\ldots,\hat{m}_i,\ldots,m_s).$$
Here $\hat{\hspace{.1cm}}$ denotes the removal of an element from a tuple,
and $(a_1,\ldots,a_n)\preceq (b_1,\ldots,b_n)$ if and only if $a_i\leq
b_i$ for all $i$.

Let
\[A(Z) = \left[\begin{array}{ccc}
\myunderbrace{c_1 ~~ \cdots ~~ c_1}{m_1} & \cdots &
  \myunderbrace{c_s ~~\cdots~~ c_s}{m_s}\end{array}\right].
\newline
\vspace{.5cm}
\]
Denote $d(Z) = d$ and $A(Z)=A$.
From the definition, we can find $d$ columns such that
$M-d$ is maximum number of columns in $A$ that span
an $n$-dimensional vector space.  Here, $M = m_1 + \cdots +m_s$.
Let $\Omega$ denote the set of these $M-d$ columns.

Let $c$ be any column of $A$, with $c\notin \Omega$.
Such a column exists, because if every column of $A$
belonged to $\Omega$, then the rank of $A$ would not be $n+1$.

Let $A'$ be the matrix obtained from $A$ after removing the column $c$.
We now consider two cases.

\noindent
{\bf Case 1.} $\rank(A')=n+1$.

\noindent
Because $c\notin\Omega$, then
$\Omega$ consists of columns of $A'$, and the columns
in $\Omega$ span an $n$-dimensional vector space. If we let
$d'$ denote the minimum distance of $A'$, we then have
$|\Omega| \leq (M-1)-d'$.  But because
$|\Omega|=M-d$, we obtain $d \geq d'+1$.

After permuting the columns of $A$, we can assume that
we have removed the first column of $A$ to construct
$A'$.  We then associate to $A'$ the fat point scheme
$Z' = (m_1-1)P_1 + m_2P_2 + \cdots + m_sP_s$.

Let $F \in (I_{Z'})_{\alpha(I_{Z'})}$
be any form of smallest degree in $I_{Z'}$, and let $L \in (I_{P_1})_1$ be
any linear form in the ideal $I_{P_1}$.  Then $FL \in
(I_Z)_{\alpha(I_{Z'})+1}$,
whence $$\alpha(I_{Z'})+1 \geq \alpha(I_Z).$$

If $m_1 \geq 2$, then by induction we have that $d' = d(Z') \geq
\alpha(Z') - m(Z')$.
Similarly, if $m_1 = 1$, then we must have $s - 1 \geq  n+1$.  This
is because if we remove the first column from $A$, the columns
of $A'$ all correspond to points in the set $\{P_2,\ldots,P_s\}$.  Since
the matrix $A'$ has rank $n+1$, we must have at least $n+1$ distinct
points
in this set.  But then by induction, we also know that $d' = d(Z')
\geq \alpha(Z') - m(Z')$.

Because $m(Z) \geq m(Z')$, when we
put together our pieces, we find the desired bound:
\[d(Z) \geq d(Z')+1 \geq \alpha(I_{Z'})-m(Z')+1 \geq \alpha(I_Z) - m(Z).\]

\noindent
{\bf Case 2.} $\rank(A')=n$.

If $\rank(A') = n$, then the column $c$ only appears in $A$ exactly once.
Furthermore, all the $M-1$ columns of $A'$ span a
$n$-dimensional vector space, and therefore $$M-1\leq M-d,$$ and so $d=1$,
because the
minimum distance must be positive.

After permuting the columns of $A$, we can assume that $c$ is the first
column $c_1$.  Moreover, the distinct columns  $c_2,\ldots,c_s$ in $A$
span an $n$-dimensional vector space. This means that the points
associated to these
columns are contained in a hyperplane defined by
a linear form $H$. So $$H^{m(Z)}\in I_2^{m_2}\cap\cdots\cap I_s^{m_s}.$$

Let $L \in I_1$ be a linear form vanishing at the point associated to
$c_1$.
Then $L\cdot H^{m(Z)}\in I_Z$, and therefore $m(Z)+1\geq \alpha(I_Z)$
which gives us $$d=1\geq \alpha(I_Z)-m(Z)$$ for this case as well.
\end{proof}

When $Z=X$ is reduced, the bound we obtained in the previous theorem can only be
attained in a very special situation.

\begin{theorem} \label{boundscor}
Let $X = \{P_1,\ldots,P_s\}$ be a reduced set of points, not all
contained in a hyperplane. Then $d(X) = \alpha(I_X) - 1$ if and only
if $s-1$ points of $X$ lie on a hyperplane.
\end{theorem}

\begin{proof}
$(\Leftarrow)$  The above theorem gives $d(X) \geq \alpha(I_X)-1$.  On the
other hand,
the set of points $X$ do not lie on a hyperplane, so $\alpha(I_X) \geq 2$.
In addition,
if $s-1$ points of $X$ lie on a hyperplane, this implies that the minimum
distance
of $A(X)$ is $1$ since the $s-1$ columns corresponding to the points on
the hyperplane
span an $n$-dimensional vector space.  So, $1 = d(X) \geq \alpha(I_X) - 1
\geq 2-1$,
which gives the desired conclusion.

$(\Rightarrow)$  Suppose that $s' < s-1$ is
the maximum number of points of $X$ that lie on a hyperplane.
Let $H$ be the linear form defining this hyperplane.  By definition, $s
-d(X) = s'$.
Pick any two points of $X$ not in this hyperplane,
and let $L$ be any linear form that vanishes at these two points.
For any of the remaining $t = s - (s'+2) \geq 0$ points, let $L_i$ be any
linear form vanishing at that point.  Then $G = H\cdot L\cdot L_1\cdots
L_t$ is a form
in the ideal of the points of $X$.  Furthermore
$$\deg G = t+2 = s - s'
= s - (s- d(X)) = d(X)=\alpha(I_X) - 1.$$  We have a contradiction
since $(I_X)_i = (0)$ for all $i < \alpha(I_X)$.
\end{proof}

\begin{example}\label{example00} When $Z$ is not reduced, one can
attain the bound in Theorem \ref{hombound} as well. Let
$P_1=[0:1:0],P_2=[1:0:0],P_3=[1:1:0],P_4=[0:0:1]$ be four points in
$\proj{2}$. Consider $Z=2P_1+2P_2+P_3+P_4$. We have $\alpha(I_Z)=3$
and $m(Z)=2$.

We have \[A(Z) =
\begin{bmatrix}
0 & 0 & 1 & 1 & 1 & 0\\
1 & 1 & 0 & 0 & 1 & 0\\
0 & 0 & 0 & 0 & 0 & 1
\end{bmatrix},\] which has $d(Z)=1$.
\end{example}

\begin{example} Since there are no restrictions on the base field
$\K$, the statement in Theorem \ref{hombound} can have applications in
coding theory; it gives a lower bound for the minimum distance of
$A(Z)$. This lower bound does not depend on $d(X)$, where
$X=\operatorname{Supp}(Z)$, as Theorem \ref{crudebounds} does.

Furthermore, the lower bound of Theorem \ref{hombound} improves the
bound of Theorem \ref{crudebounds} in certain cases. For example, let
$X = \{P_1,\ldots,P_6\}$ be six points in $\mathbb{P}^2$ where
$P_4,P_5,P_6$ all line on a line, and none of $P_1,P_2$, or $P_3$ lie
on this line, and moreover, there is no line that passes through
these three points.  By our choice of points, $hyp(X) = 3$, and thus
$d(X) = 6 -3 = 3$.

Now consider the fat point scheme $Z=5P_1+5P_2+5P_3+P_4+P_5+P_6.$  We
have $m_1=m_2=m_3=m(Z)=5$, and
$m_4=m_5=m_6=1$. Theorem \ref{crudebounds} gives the lower bound $d(Z)
\geq m_4+m_5+m_6 =
3$. However, for this set of fat points, $\alpha(I_Z) = 9$, whence by
Theorem \ref{hombound} we have $d(Z) \geq \alpha(I_Z)-m(Z)=9-5=4$.
So, the lower bound of Theorem \ref{hombound} is better in this case.
\end{example}


\section{The minimum socle degree of a homogeneous fat point scheme}

Let $\K$ be a field of characteristic zero. Let
$Z=mP_1+\cdots+mP_s\subset\proj{n}$
be a homogeneous fat point scheme with
$X=\operatorname{Supp}(Z)=\{P_1,\ldots,P_s\}$
not contained in a hyperplane. With the notations above, in this
section we prove
one of the main results of the paper: $$s_n(Z)\leq md(X).$$
(Note that we assume $\K$ has characteristic zero so that we can
make use of a result found in \cite{gmv} on separators of fat points.
In particular, the result that we require from
\cite{gmv} is based upon a mapping cone construction of a
graded minimal free
resolution of $I_Z$;  the maps
that appear in this construction may change if we consider
a field of characteristic $p > 0$.  A careful analysis of
\cite{gmv} would be required to determine if the results of this section
still hold in non-zero characteristics.)

Before we prove this result, we will make some remarks on the results
obtained in the
previous section for the case of homogeneous fat points. First, as observed in
Theorem \ref{crudebounds}, $d(Z)=md(X)$. Second, the result in Theorem
\ref{hombound}
is immediate in this case: if $f\in I_X$ is of degree $\alpha(I_X)$,
then $f^m\in I_Z$,
and hence $m\alpha(I_X)\geq \alpha(I_Z)$. With $d(Z)=md(X)$ and
$d(X)\geq \alpha(I_X)-1$,
we indeed obtain that $d(Z)\geq \alpha(I_Z)-m$.

An interesting question remains: when is this bound attained? A simple
computation
 shows that the bound is attained whenever $d(X)=1$ and
$\alpha(I_Z)=m\alpha(I_X)$.
Now, looking at Example \ref{example00}, with $Z=mP_1+mP_2+mP_3+mP_4$,
when $m=2$ we
have $\alpha(I_Z)=4=2\alpha(I_X)$, whereas when $m\geq 3$,
$\alpha(I_Z)\leq 2m-1$.

In general there is no control on $\alpha(I_Z)$ as we vary $m$. As we
can see in the
second part of the example below, some ``random'' behavior happens in
general for
$s_n(Z)$ when $Z$ is a homogeneous fat points scheme.

\begin{example} \label{attainedbounds}
Let $X= \{P_1,P_2,P_3,P_4,P_5\} \subseteq \mathbb{P}^2$ where
$P_1 =[1:0:0],P_2 = [0:1:0],P_3 = [0:0:1],P_4 = [1:1:0],$ and $P_5 =
[1:3:1]$.
Set $Z=mP_1+mP_2+mP_3+mP_4+mP_5$, with $m\geq 1$. We have that
$d(X)=2\geq 2 =\alpha(I_X)$, and $d(Z)=md(X)=2m$.  For $m = 1,\ldots,7$,
we  calculate the minimum socle degree of $I_Z$:

\begin{center}
\begin{tabular}{c|c|c|c|c|c|c|c}
$m$&1&2&3&4&5&6&7\\
\hline
$s_2(Z)$&2&4&6&8&10&12&14
\end{tabular}
\end{center}

Let $X = \{P_1,P_2,P_3,P_4\} \subseteq \mathbb{P}^2$ where
$P_1=[1:0:0],P_2=[0:1:0],P_3=[0:0:1]$, and $P_4=[1:1:0]$, and set
$Z=mP_1+mP_2+mP_3+mP_4$ with $ m\geq 1$.
We have $d(X)=1=2-1=\alpha(I_X)-1$ and $d(Z)=md(X)=m$.
For $m =1,\ldots,7$, we calculate $s_2(Z)$:

\begin{center}\begin{tabular}{c|c|c|c|c|c|c|c}
$m$&1&2&3&4&5&6&7\\
\hline
$s_2(Z)$&1&3&5&6&8&10&11
\end{tabular}
\end{center}
\end{example}

Before we state and prove Theorem \ref{maintheorem}, we shall need the
notion of a separator of a fat point,
as found in \cite{gmv}.

\begin{definition}
Let $Z = m_1P_1 + \cdots +m_sP_s \subseteq \proj{n}$ be a
set of fat points, and suppose that $Z' = m_1P_1 + \cdots + (m_i-1)P_i
+ \cdots + m_sP_s$ for some $i =1,\ldots,s$.  (If $m_i = 1$, we simply
omit the point $P_i$ in $Z'$.)  We call $F \in R = \K[x_0,\ldots,x_n]$
a {\it separator of $P_i$ of multiplicity $m_i$} if
$F \in I_{Z'} \setminus I_Z.$
\end{definition}

When all of the $m_i$s equal one in the above definition, we recover
the definition of a separator of a reduced point as found in
\cite{abm,b,gmr}.  We now apply \cite[Theorem 5.4]{gmv}.

\begin{theorem}\label{fatpointsocle}
Let $Z = m_1P_1 + \cdots +m_sP_s \subseteq \proj{n}$ be a
set of fat points, and suppose that $Z' = m_1P_1 + \cdots + (m_i-1)P_i
+ \cdots + m_sP_s$ for some $i =1,\ldots,s$.  Let $F$ be
any separator of $P_i$ of multiplicity of $m_i$ of smallest degree,
i.e., if $F'$ is any other separator of $P_i$ of multiplicity
$m_i$, then $\deg F' \geq \deg F$.  Then
\[\deg F \geq s_n(Z).\]
\end{theorem}

\begin{proof}  Note that \cite[Theorem 5.4]{gmv} actually proves
something stronger: if $\overline{G}$ is any
minimal generator of the $R$-module $I_{Z'}/I_Z$, then $\deg G + n$
appears
as a shift  in the last module in the graded minimal free resolution of
$I_Z$.  Because
$\overline{F}$ will be a minimal generator of $I_{Z'}/I_Z$
of smallest degree, $\deg F + n$ will appear as
a shift in the last module in the graded minimal free resolution, and thus
$\deg F +n - n \geq
s_n(Z).$
\end{proof}

We need one other result.
In Remark \ref{geometric}, $hyp(X)$ denotes the maximum number of points
of
a reduced set of points $X$ contained in some hyperplane. To obtain the
maximum number of points of $X$ contained in some hypersurface of degree
$a$,
by \cite{mp}, one should compute $hyp(v_a(X))$, where $v_a$ is the
Veronese embedding
of degree $a$ of $\mathbb P^n$ into $\mathbb P^{N_a}$, where $N_a=
{{n+a}\choose{a}}-1$.
Let us denote
\[d(X)_a=|X|-hyp(v_a(X)).\]
Observe that $d(X)_1=d(X)$.

\begin{remark}
As an aside,
$d(X)_a$ is the minimum distance of the evaluation code $\mathcal C(X)_a$
(see \cite{h,tvn} for more details).
However, we will not need this interpretation.
\end{remark}

The following lemma will then constitute a key tool needed to prove our
main result:

\begin{lemma}[{\cite[Proposition 2.1]{t1}}]\label{lem:recursion}
Let $X = \{P_1,\ldots,P_s\} \subseteq \proj{n}$ be a set of distinct
reduced points.
If $d(X)_b\geq 2$ for some $b$, then for all
$1\leq a\leq b-1$, we have $d(X)_a\geq d(X)_{a+1}+1$ and therefore
$d(X)_a\geq b-a+2$.
\end{lemma}

\begin{remark} Here is an intuitive, geometrical proof of
the above lemma. Suppose that $s - d(X)_a$ points of $X$ lie on a
hypersurface of degree $a$.
Then we need to have at least
$s-d(X)_a+1$ points lying on a hypersurface of degree $a+1$.  Indeed,
if you take the hypersurface $V(F)$ of degree $a$ containing the
$s-d(X)_a$
points, and any hyperplane $V(L)$ through one of the remaining points,
then the hypersurface $V(F\cdot L)$ of degree $a+1$ will contain
$s-d(X)_a+1$ points.
So, $d(X)_{a+1} \leq s -  (s-d(X)_a+1)$, i.e., $d(X)_{a+1} +1 \leq
d(X)_a$.
\end{remark}

We come to our main theorem.

\begin{theorem}\label{maintheorem}
Let $Z = mP_1 + \cdots +mP_s
\subseteq \proj{n}$ be a homogeneous set of fat points with
$X = \operatorname{Supp}(Z)$ not contained in a hyperplane.
\begin{enumerate}
\item[$(i)$] If $d(X) \geq \alpha(I_X)$, then $s_n(Z)\leq md(X)$.
\item[$(ii)$] Otherwise, if $d(X) =  \alpha(I_X) -1$, then $s_n(Z)\leq 2m-1.$
\end{enumerate}
\end{theorem}

\begin{proof}
From Theorem \ref{crudebounds}, we have that $d(Z)=md(X)$.

For each $P_i \in X$, let
\[\delta_i = \min\{\deg F ~|~ F(Q) = 0 ~~\mbox{for each $Q \in X
\setminus\{P_i\}$ but $F(P_i) \neq 0$}\}.\]
Let $\delta = \min_{i=1}^s\{\delta_i\}$ be the minimum degree.
After relabeling, we may
assume that $\delta=\delta_1$. Let $X'=X\setminus\{P_1\}$ and let
$F \in I_{X'}\setminus I_X$ be a separator of degree $\delta$.

Let $G \in (I_X)_{\alpha(I_X)}$. Then $$F\cdot G^{m-1} \in I_W\setminus
I_Z,$$
where $W=(m-1)P_1+mP_2+\cdots+mP_s$.   In other words, $F\cdot G^{m-1}$ is
separator
of $P_1$ of multiplicity $m$.
So $$\delta+(m-1)\alpha(I_X)\geq\Delta(Z),$$
where $\Delta(Z)$ is the smallest degree of
a separator of $P_1$ of multiplicity $m$.
By Theorem \ref{fatpointsocle}, we then have
$$\delta + (m-1)\alpha(I_X) \geq \Delta(Z)\geq s_n(Z).$$

$(i)$ If $d(X)\geq \alpha(I_X)$, then
$\delta\geq 2$. Otherwise, if $\delta=1$, then $s-1$ points of $X$ will
lie in a
hyperplane. But by Corollary \ref{boundscor}, this can only happen if
$d(X) = \alpha(I_X)-1$.

Also, $d(X)_{\delta-1} \geq 2$.  If $d(X)_{\delta-1} \leq 1$, then there
is a hypersurface
of degree $\delta-1$ that contains either all the points of $X$, or all
but one point of $X$.  But this would contradict our choice of
$\delta$; it is the smallest degree of a form that passes through
all the points of $X$ except one.
So, by Lemma \ref{lem:recursion}  with $b=\delta-1$ and $a=1$, we have
$d(X)\geq \delta$.  With this fact, and since $d(X)\geq \alpha(I_X)$,
we obtain
$$d(Z)=md(X)=d(X)+(m-1)d(X)\geq \delta+(m-1)\alpha(I_X) \geq\Delta(Z)\geq
s_n(Z).$$

$(ii)$  If $d(X) = \alpha(I_X)-1$,
then $s-1$ points of $X$ lie on a hyperplane by Corollary \ref{boundscor}.
As shown
in the proof of this corollary, this also implies $d(X) = 1$.
Let $V(H)$ be the hyperplane through the $s-1$ points,
and let $L$ be a linear form that vanishes at the remaining point off the
hyperplane (say $P_1$).
Then
\[H^m\cdot L^{m-1} \in I_W \setminus I_Z\]
where $W = (m-1)P_1+mP_2 + \cdots + mP_s.$  Hence, $H^m\cdot L^{m-1}$
is a separator of $P_1$ of multiplicity $m$.
Thus, by Theorem \ref{fatpointsocle},
\[2m-1 \geq s_n(Z).\] \end{proof}

\begin{remark} Looking at Example \ref{attainedbounds}, observe that the first part of this example gives that for each $m =1,\ldots,7$ the lower bound of Theorem
\ref{maintheorem}(i) is attained. The second part shows that for $m=1,2$ and $3$ the lower bound of Theorem \ref{maintheorem}(ii) is also attained.
\end{remark}

We end this section with a question based upon our results.

\begin{question} Can we generalize the lower bound of Theorem
\ref{maintheorem} to
non-homogeneous fat points?   Is it true that $d(Z)\geq s_n(Z)-m(Z)+1,$
where
$m(Z)$ is the maximum multiplicity of a point in $Z$, for any $Z$?
In other words, because $s_n(Z) \geq \alpha(I_Z)-1$,
can Theorem \ref{hombound} be improved to $d(Z) \geq s_n(Z) - m(Z) +1 \geq
\alpha(I_Z) -m(Z)$?
\end{question}


\section{A case study: complete intersections}

Reduced matrices of the from $A(X)$ were studied by Hansen \cite{h} and
Gold, Little, and Schenck \cite{gls}.  In both cases, the
authors focused on the case that the associated set of reduced points $X$
was
a complete intersection. (Their results were later generalized in
\cite{t1} to case that $X$ was Gorenstein, and in \cite{t3} to the case that $X$
was any reduced set of points.)
Building upon their work, we consider
matrices of the form $A(Z)$ when the support of the fat points
$Z$ is a complete intersection.

Recall that a set of points of $X \subseteq \proj{n}$
is a {\it complete intersection of type $(d_1,\ldots,d_n)$}
if there exists a regular sequence of homogeneous forms
$F_1,\ldots,F_n \in R$
with $\deg F_i = d_i$ such that $I_X = (F_1,\ldots,F_n)$.
We usually denote $X$ by $CI(d_1,\ldots,d_n)$.
Because the $F_i$'s defining a complete intersection are
homogeneous, any permutation of the $F_i$'s is also a complete
intersection.
So, we can make the assumption that $d_1 \leq d_2 \leq \cdots \leq d_n$.
Because we are interested in the case that $A(X)$ has full rank,
we can also assume that $2 \leq d_1$.  If $d_1 =1$, then the set
of points $X$ would be contained in a hyperplane.

Recall that the theme of this paper is to study the shifts in the
graded minimal free resolution in terms of the minimum distance of
$A(Z)$. While some bounds can be found, the following
example shows that this will not be enough.

\begin{example}\label{ci} Consider the following two sets of points
$X_1,X_2\subseteq \proj{2}$, both examples of complete intersections
of the form $CI(2,3)$.  In the first case
the conic is irreducible, while in the second case the conic is reducible.
\begin{center}
\epsfig{file=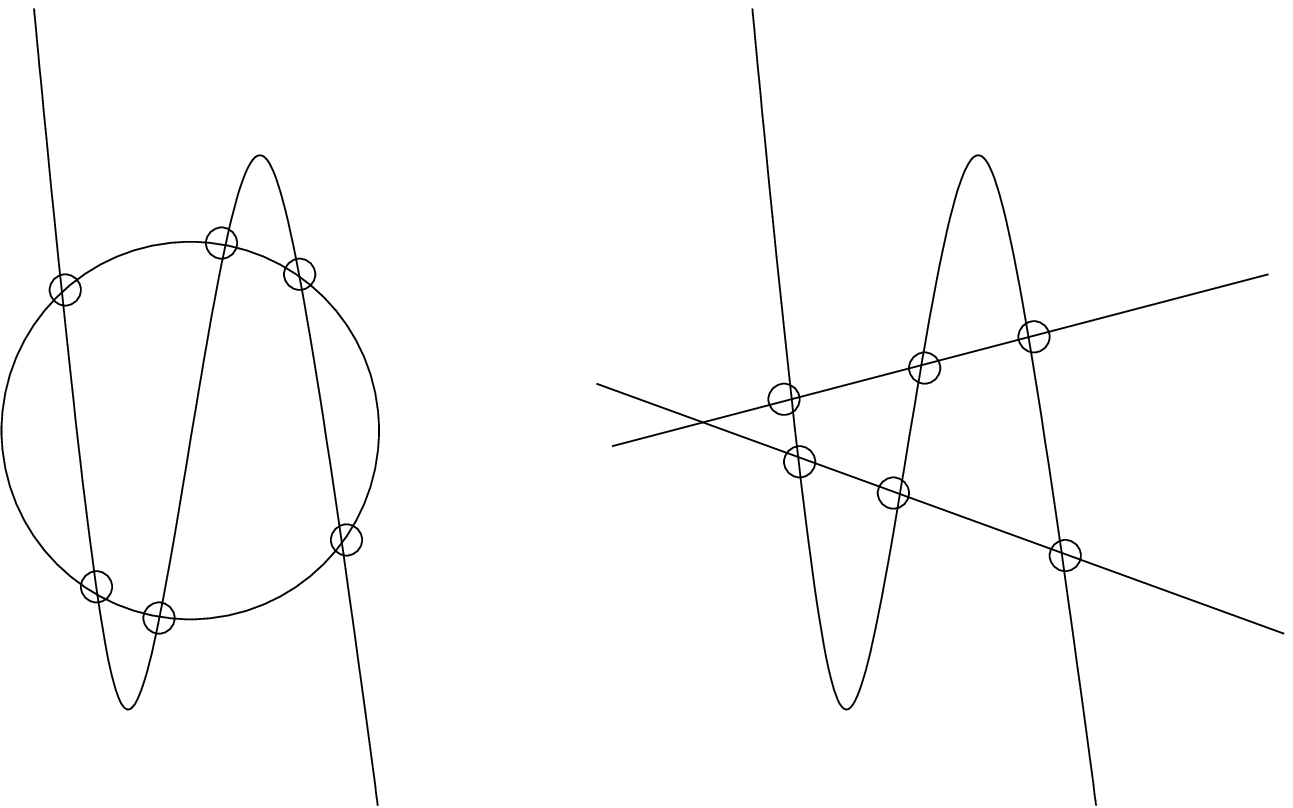,height=2in,width=5in}
\end{center}
The graded minimal free resolutions of $R/I_{X_1}$ and
$R/I_{X_2}$ are the same, i.e.,
$$0\rightarrow R(-5)\rightarrow R(-2)\oplus R(-3)\rightarrow R\rightarrow
R/I\rightarrow 0
~~\mbox{with $I=I_{X_1}$ or $I_{X_2}$.}$$

We have $s_2(X_1)=s_2(X_2)=5-2=3$. If $A_1$ and $A_2$ are the
corresponding
matrices (i.e., the columns of these matrices will be given by the
homogeneous
coordinates in the algebraic closure of $\K$ of the points), we see from
Remark \ref{geometric}, that $d(A_1)=6-2=4$, and $d(A_2)=6-3=3$.

Indeed the bound in Theorem \ref{maintheorem}, with $m=1$, is satisfied
for
both cases, that is, $d(A_i)\geq s_2(A_i)$ for $i=1,2$.  But this example
also shows that
one cannot rely on the graded minimal free resolution
alone to find the minimum distance.
\end{example}

\subsection{Homogeneous fat points with complete intersection support}
As shown in Theorem \ref{maintheorem}, we can bound $d(Z)$ in terms of
$s_n(Z)$.  In the case that $\operatorname{Supp}(Z)$ is a complete
intersection, we can get an explicit value for $s_n(Z)$ when all the
multiplicities are equal.

\begin{lemma}\label{soclevalueci}
Let $Z = mP_1 + \cdots + mP_s \subseteq \proj{n}$ be a set of fat
points with $\operatorname{Supp}(Z) = CI(d_1,d_2,\ldots,d_n)$.
Then
\[s_n(Z) = md_1+d_2+d_3+ \cdots + d_n - n.\]
\end{lemma}

\begin{proof}
The defining ideal of $I_Z$ is $I_X^m$ with $X = \operatorname{Supp}(Z)$.
But $I_X^m$ is  a power of a complete intersection, so
one can use the formula of \cite[Theorem 2.1]{gv}.
\end{proof}

If we want to see for what such $Z$ are the bounds in
Theorem \ref{maintheorem} attained, we obtain

\begin{theorem}  Let $X = CI(d_1,\ldots,d_n) \subseteq \mathbb{P}^n$
with $n \geq 2$,
and let $Z$ be the
homogeneous
set of fat points of multiplicity $m$ whose support is $X$.  Then
\[md(X) = s_n(Z)\mbox{ if and only if } X=CI(2,2).\]
\end{theorem}

\begin{proof} First, we show that we can exclude the second part of
Theorem \ref{maintheorem} from consideration.  Indeed, suppose
that  $d(X)=\alpha(I_X)-1$. By the proof of Corollary \ref{boundscor},
we must have $\alpha(I_X)=d_1=2$, and thus Theorem \ref{soclevalueci}
implies  that $s_n(Z) = 2m + d_2 + \cdots + d_n - n$. From Theorem \ref{maintheorem}, we have $s_n(Z)\leq 2m-1$, and therefore, $d_2+\cdots +d_n-n \leq -1$. But $2= d_1 \leq d_2 \leq \cdots \leq d_n$ implies that
$2(n-1)-n \leq d_2 + \cdots + d_n-n \leq -1$, thus giving us $n-2 \leq -1$,
that is, $n \leq 1$, thus giving the contradiction. Thus, we can assume that $d(X) \geq \alpha(I_X) = d_1$.

Suppose $s_n(Z)=md(X)$. From Theorem \ref{maintheorem},
$m=1$, we have $d(X)\geq s_n(X)=d_1+\cdots+d_n-n$ and,
from Lemma \ref{soclevalueci}, $s_n(Z)=md_1+d_2+\cdots+d_n-n$.

Denote $s_n(Z) = U$ and $ms_n(X) = V$.
 We then have $U = (m-1)d_1+V/m$.
Now we will have
$U \geq V$ if and only if $(m-1)d_1+V/m \geq V$.
But this inequality is equivalent to
  \[m(m-1)d_1 \geq (m-1)V.\]
So $md_1 \geq V$. But $V = ms_n(X)$, so
$d_1 \geq s_n(X) = d_1+d_2+\cdots+d_n-n$.
We thus get $U \geq V$ if and only if $n \geq d_2+\cdots+d_n$.
But $2 \leq d_2 \leq \cdots \leq d_n$, so we have
$n \geq 2(n-1)$, i.e., $2 \geq n$.  But this forces
$n =2, d_2 = 2,$ and $d_1=2$.

Thus, unless $X = CI(2,2)$, we have  $ms_n(X) > s_n(Z)$, and therefore
$md(X)\neq s_n(Z)$.

\vskip .1in

If $X = CI(2,2)$, then $s_2(Z) = m\cdot2+2-2 = 2m$.  Since $d(X)=2$ we
have indeed that $s_2(X)=md(X)$.
\end{proof}

Lemma \ref{soclevalueci} also lets us recover a result of Gold,
Little, and Schenck (\cite{gls}) as a corollary;
their result is the case when all the multiplicities equal one.

\begin{corollary}\label{cibound}
Let $X = \{P_1,\ldots,P_s\} \subseteq \proj{n}$ be a reduced set of points.
If $X = CI(d_1,\ldots,d_n)$, then
\[d(X) \geq  d_1+d_2+d_3+ \cdots + d_n - n.\]
\end{corollary}

\begin{proof}
By Theorem \ref{maintheorem}, $d(X) \geq s_n(X)$.  Now use Theorem
\ref{soclevalueci}.
\end{proof}

\subsection{B\'ezout's Theorem.} It is known that the
bound of Corollary \ref{cibound} is
far from optimal.
For complete intersections of the form $X = CI(d_1,\ldots,d_n) \subseteq
\proj{n}$,
we will use B\'ezout's Theorem to improve known bounds.

There are many ways one can state B\'ezout's Theorem in $\proj{n}$.
The version we shall use can be found in Chapter 6.2 of \cite{g}. We
thank the anonymous referee for pointing out that this version is
valid also when $\K$ is algebraically closed field of positive
characteristic.

We first recall that the {\it degree}
of a scheme $W \subseteq \proj{n}$, denoted
$\deg(W)$, is defined to be $(\dim W)!$ times the leading
coefficient of the Hilbert polynomial of $W.$

\begin{theorem}[B\'ezout's Theorem]  Let $X$ be a projective subscheme
of $\proj{n}$ with $\dim X \geq 1$.  If $f \in \K[x_0,\ldots,x_n]$ is
a homogeneous form such that no component of $X$ is contained in $V(f)$,
the variety defined by $f$, then
\[\deg(X \cap V(f)) = \deg(f)\cdot\deg(X).\]
\end{theorem}

To make use of this theorem, we recall two standard facts:
\begin{enumerate}
\item[$\bullet$] If $W$ is a reduced finite set of points, then $\deg(W) = |W|$.
\item[$\bullet$] If $W = CI(d_1,\ldots,d_r)$, then $\deg(W) = d_1\cdots
d_r$.
\end{enumerate}

First, a general result:

\begin{theorem}\label{genci} Let $Y$ be a curve in $\mathbb P^n$ with no
component
contained in a hyperplane. Let $V(f)$ be a hypersurface of degree $a>1$
such that
$X=Y\cap V(f)$ is a reduced zero-dimensional scheme. Then $X$ has
minimum distance
\[d(X)\geq (a-1)\deg(Y).\]
\end{theorem}

\begin{proof} Suppose that $Y$ has a component $W$ contained in $V(f)$.
Then
$W\subseteq X$. Since $\dim(X)=0$, then $\dim(W)=0$, and so
$W=W_1\cup\cdots\cup W_m$,
where each $W_i$ is set-theoretically a point in $\mathbb P^n$. Since a
point is
always contained in a hyperplane, we have contradicted our assumption that
$Y$ has no component
in a hyperplane. So we can apply B\'ezout's Theorem to obtain
$$|X|=\deg(X)=a\cdot\deg(Y).$$

We have that $d(X)=|X|-hyp(X)$, so it suffices to show that $hyp(X)\leq
\deg(Y)$.
Suppose that $h=hyp(X)>\deg(Y)$ and that $V(L)$ is the hyperplane
containing the $h$ points of $X$.
Since no component of $Y$ is contained in $V(L)$, then $\dim(Y\cap
V(L))=0$ and furthermore we
can apply B\'ezout's Theorem once more to obtain that
$$\deg(Y\cap V(L))=\deg(L)\cdot\deg(Y)=\deg(Y),$$ since $\deg(L)=1$.
Since $X\subset Y$, then $X\cap V(L)\subseteq Y\cap V(L)$. Therefore the
$h$ points of
$X$ lying on $V(L)$ should be contained in $Y\cap V(L)$.
But then $\deg(Y\cap V(L))\geq h$, which contradicts the assumption that
$h>\deg(Y)$.
\end{proof}

\begin{example} We can construct sets of points such that the bound in
Theorem \ref{genci}
is attained. Let $Y\subset\mathbb P^n$ be an irreducible curve, not
contained in an
hyperplane. Let $g$ be a form of degree $a-1\geq 1$ and let $L$ be a
linear form such
that $X=V(L\cdot g)\cap Y$ is a reduced zero-dimensional scheme.

Since $V(L\cdot g)=V(L)\cup V(g)$, then $V(L)\cap Y\subseteq X$, is a
reduced
zero-dimensional scheme of degree $\deg(V(L)\cap Y)=\deg(Y)$. So the
hyperplane
$V(L)$ contains $\deg(Y)$ points of $X$. This implies that
$hyp(X)\geq \deg(Y)$. But from Theorem \ref{genci} we have
$hyp(X)\leq \deg(Y)$, and therefore we get an equality.
\end{example}

As a corollary, we improve the bound on $d(X)$ when $X$ is complete
intersection
with an additional condition.

\begin{corollary}\label{bezout} Let $X = CI(d_1,\ldots,d_n) \subseteq
\proj{n}$, with
$2\leq d_1\leq \cdots \leq d_n$. If $I_X= (F_1,\ldots,F_n)$, then for each
$i=1,\ldots,n$
let $X_i$ be the complete intersection
$CI(d_1,\ldots,\hat{d_i},\ldots,d_n)$, with ideal
$I_{X_i}=( F_1,\ldots,\hat{F_i},\ldots,F_n)$.  In addition,
suppose that there exists an index $j \in \{1,\ldots,n\}$ such that $X_j$
has no
component contained in a hyperplane.
Then
\[d(X) \geq  (d_1-1)d_2d_3\cdots d_n .\]
\end{corollary}

\begin{proof}  Because $X$ is a reduced complete intersection, $|X| =
d_1\cdots d_n$. Also, for each $i$, $X_i$ is a complete intersection curve
of degree $\deg(X_i)=d_1\cdots\hat{d_i}\cdots d_n$.

Let $j$ be the index such that $X_j$ has no component contained in a
hyperplane. From
Theorem \ref{genci}, with $Y=X_j$, and $f=F_j$, we obtain
$$d(X)\geq (d_j-1)d_1\cdots\hat{d_j}\cdots d_n.$$
Since $d_1\leq d_2\leq\cdots\leq d_n$, then
$d_2\cdots d_n\geq d_1\cdots \hat{d_j}\cdots d_n$. Hence, the assertion.
\end{proof}

We expect that the hypothesis in Corollary \ref{bezout} that there exists
an $X_j$ with no
component contained in a hyperplane can be dropped.  We make the following
conjecture:

\begin{conjecture}\label{conjecture} Let $X = CI(d_1,\ldots,d_n) \subseteq
\proj{n}$, with
$2\leq d_1\leq \cdots \leq d_n$.  Then $d(X) \geq  (d_1-1)d_2d_3\cdots d_n$.
\end{conjecture}

When $n=2$, we only need B\'ezout's Theorem for curves to prove Conjecture
\ref{conjecture}.

\begin{theorem}\label{n=2}
  Let $X = CI(d_1,d_2) \subseteq \proj{2}$, with
$2\leq d_1\leq d_2$.  Then $d(X) \geq  (d_1-1)d_2$.
\end{theorem}

\begin{proof}  Let $h = hyp(X)$, and suppose that $h > d_2$.  Therefore,
there is a line (since we are in $\mathbb{P}^2$) that contains $h > d_2$
points
of $X$.  Let $L$ be the form that defines this line, and suppose that
$I_X = (F_1,F_2)$.  B\'ezout's Theorem for curves in $\mathbb{P}^2$
implies
that $L|F_1$ since $V(L) \cap V(F_1)$ meet at $h > d_1$ points.
Similarly,
$L|F_2$ since $V(L) \cap V(F_2)$ meet at $h > d_2$ points.  But
then $L$ divides gcd$(F_1,F_2)$, contradicting the assumption
that $F_1$ and $F_2$ form a regular sequence.
\end{proof}

\begin{remark}
The bound of Theorem \ref{n=2} improves the bound of Gold, Little, and
Schenck \cite{gls}
(see Corollary \ref{cibound})
when $3 \leq d_1$.  Indeed,
we have $d_1d_2 - d_2 > d_1+d_2 -2$ if and only if
$d_1d_2-d_1 = d_1(d_2-1) > 2(d_2-1) = 2d_2 -2$ if and only if $d_1 \geq
3$.  When
$2=d_1$, the two bounds are the same, i.e, $(d_1-1)d_2 = d_1+d_2-2 = d_2$.
\end{remark}


\bibliographystyle{amsalpha}

\end{document}